\documentclass{amsart}
\usepackage{chngcntr}
\usepackage{apptools}
\usepackage{color}
\usepackage{amsthm,amssymb,verbatim}
\usepackage{graphicx}
\usepackage{enumerate}
\usepackage{chngcntr}
\usepackage{apptools}
\usepackage{color}
\usepackage{amsthm,amssymb,verbatim}
\usepackage{graphicx}
\usepackage{enumerate}

\newcommand{\talpha}{\tilde\alpha}

 \newcommand{\inj}{\operatorname{inj}}    %injectivity radius

  \newcommand{\MMM}{\mathbf M}

  \newcommand{\Riem}{\operatorname{Riem}}
  
  \newcommand{\Tt}{\mathcal{T}}
 
 \newcommand{\RR}{\mathbf{R}}  % reals
    \newcommand{\dist}{\operatorname{dist}}

 \newcommand{\eps}{\epsilon}
 \newcommand{\Tan}{\operatorname{Tan}}
 \newcommand{\Length}{\operatorname{Length}}
 
 \newcommand{\tbeta}{\tilde \beta}

 \newcommand{\Xx}{\mathcal{X}}
 \newcommand{\Yy}{\mathcal{Y}}

\newcommand{\vv}{\mathbf v}
\newcommand{\ww}{\mathbf w}
\newcommand{\ee}{\mathbf e}

\newcommand{\spt}{\operatorname{spt}}
\newcommand{\uu}{\mathbf{u}}

\usepackage{amsthm}

\input epsf
\def\begfig {
\begin{figure}
\small }
\def\endfig {
\normalsize
\end{figure}
}

    \newtheorem{theorem}    {Theorem}   %    [section]
    \newtheorem{lemma}      [theorem]       {Lemma}
    \newtheorem{corollary}  [theorem]     {Corollary}
    \newtheorem{proposition}       [theorem]       {Proposition}
    
    \newtheorem{claim}{Claim}
    \newtheorem*{theorem*}{Theorem}
    \theoremstyle{definition}
    \newtheorem{definition}  [theorem] {Definition}
     
    \theoremstyle{definition}
    \newtheorem{remark}   [theorem]       {Remark}
    \newtheorem{case}   {Case}
\title[The Avoidance Principle]{The Avoidance Principle for Noncompact Hypersurfaces Moving by Mean Curvature Flow}
\author{Brian White}
\address{Department of Mathematics\\ Stanford University\\ Stanford, CA 94305}
\date{1 February, 2024.  Revised 24 March, 2024.}
\usepackage{hyperref}
\usepackage{enumerate}
\usepackage[alphabetic, msc-links, backrefs]{amsrefs}
\begin{document}
\maketitle
\begin{abstract}
Consider a pair of smooth, possibly noncompact, properly immersed hypersurfaces moving by mean curvature flow, 
or, more generally, a pair of weak set flows.  We prove that if the ambient space is Euclidean space
and if the distance between the two surfaces is initially nonzero, then the surfaces remain disjoint
at all subsequent times.

We prove the same result when the ambient space is a complete Riemannian manifold of
nonzero injectivity
radius, provided the curvature tensor (of the ambient space) and all its derivatives are bounded.
\end{abstract}

\section{introduction}
The classical avoidance principle for mean curvature flow says that
two smooth, properly immersed, initially disjoint hypersurfaces moving by mean curvature flow in Euclidean space
remain disjoint as long as they are smooth, if at least one of them is compact.  
That avoidance principle is an easy consequence
of the strong maximum principle.   Ilmanen~(\cite{ilmanen-elliptic}*{\S10}, \cite{ilmanen-proc})
 generalized the avoidance principle
to arbitrary ``set-theoretic subsolutions of mean curvature flow'' 
or (in the terminology of~\cite{white-topology} and~\cite{hershkovits-w-avoid})
``weak set flows".  A special case is that of surfaces moving by the 
level set flow of~\cite{chen-giga-goto} and~\cite{evans-spruck}.
The support of a codimension-one, integral Brakke flow is a weak set flow, so Ilmanen's avoidance principle
also applies to such Brakke flows.
In~\cite{hershkovits-w-avoid}, Ilmanen's avoidance principle for weak set flows in Euclidean space was extended to weak set flows
in a complete Riemannian manifold, provided  the Ricci curvature of the ambient manifold is bounded below.

Those papers left open, even in the case of smooth hypersurfaces in Euclidean space,
 the question of whether there is an avoidance principle when neither surface is compact.
The correct hypothesis for such an avoidance principle
 is not that the surfaces are initially disjoint.  There is, for example,
a smooth curvature flow $t\in [0,T)\mapsto C(t)$ in the plane such that at time $0$ the curve is the union of the graphs
$y= 1/(1+x^2)$ and $y= -1/(1+x^2)$ and such that at times $t\in (0,T)$, the curve is a simple closed curve. 
(See~\cite{ilmanen-indiana}*{7.3}.)
Of course the static flow $t\mapsto X(t):=\RR\times \{0\}$ is also a curvature flow.
Note that $C(t)$ and $X(t)$ are disjoint at time $0$ but intersect for all $t\in (0,T)$.

The correct hypothesis is that the distance between the two surfaces is initially positive.
Note that if either surface is compact, this hypothesis is equivalent to disjointness. 

This paper proves that the avoidance principle does indeed remain true for noncompact
weak set flows in Euclidean space, or, more generally, in a Riemannian manifold, provided the manifold
is well-behaved at infinity:

\begin{theorem}[Avoidance Theorem]
Suppose that $N$ is a complete, connected Riemannian manifold
with positive injectivity radius  such that $|\nabla^k \Riem|$ is bounded  for each nonnegative
 integer $k$.
Let $\Lambda$ be a lower bound for the Ricci curvature of $N$.
Suppose that $t\in [0,\infty)\mapsto X(t), Y(t)$ are weak set flows in $N$.
Then
\[
   e^{-\Lambda t} d(X(t),Y(t)) 
\]
is an increasing function of $t$.
\end{theorem}

Here $d(P,Q)= \inf \{d(p,q): p\in P,\, q\in Q\}$, where $d(p,q)$ is geodesic distance
from $p$ to $q$. (Thus $d(p,q)=|p-q|$ in Euclidean space.)

We now describe Ilmanen's proof that the avoidance principle holds
in Euclidean space if at least one of the
flows is compact.  By definition of weak set flow, a weak set flow
cannot bump into a smooth, compact mean curvature flow, provided the two flows are initially disjoint.
  It follows easily (in Euclidean space) 
 that the distance between the two is an increasing
function of time.  Suppose that 
 two weak set flows $X(\cdot)$ and $Y(\cdot)$,
one compact, are initially disjoint.  
Ilmanen proved a $C^{1,1}$ interpolation theorem, according to which there is a compact $C^{1,1}$ hypersurface $M$
between $X(0)$ and $Y(0)$ such that 
\begin{equation}\label{distant}
d(X(0),M)=R=d(Y(0),M), 
\end{equation}
where $R=\frac12 d(X(0),Y(0))$.
Now let $M$ evolve by mean curvature flow.  Then $M(t)$ will be smooth for $t\in (0,\eps]$
for some $\eps>0$.   It follows that $d(X(t),M(t))\ge R$ and $d(Y(t),M)\ge R$
for all $t\in [0,\eps]$, and therefore that 
\[
   d(X(t),Y(t)) \ge d(X(t),M(t)) + d(M(t),Y(t)) \ge 2R
\]
for all $t\in [0,\eps]$.  The result follows rather directly.

The same idea was used in~\cite{hershkovits-w-avoid} for general ambient manifolds.
That paper replaced Ilmanen's $C^{1,1}$ interpolation theorem by an easier
$C^1$ interpolation theorem.

If one tries to adapt Ilmanen's proof  to noncompact $X(t)$ and $Y(t)$, one can at least get started:
there is a locally $C^{1,1}$ surface $M$ between $X(0)$ and $Y(0)$ satisfying~\eqref{distant}. 
 However, it is not clear that there is 
 uniformly $C^{1,1}$ interpolating surface, and it seems that uniform $C^{1,1}$ bounds
are needed in order to use it to prove avoidance.  

Likewise, to adapt the proof in~\cite{hershkovits-w-avoid} to noncompact $X(t)$ and $Y(t)$, 
it would seem that uniform $C^1$ bounds on the interpolating surface $M$ are needed.
However, this paper shows that uniform $C^1$ bounds are not required.

Let $K$ be the set of points in $N$ at distance $\ge R$ from $X(0)\cup Y(0)$, where 
$R= (1/2)d(X(0),Y(0))$.  Thus the interpolating surface $M$ of \cite{hershkovits-w-avoid}
 is contained in $K$.   This paper does not prove uniform $C^1$ bounds on $M$.
 However, it does prove uniform $C^1$ bounds on the  points of $M$ near $\partial K$,
 and it shows that those bounds suffice for  proving avoidance.
 
It is natural to wonder if there is a simpler proof in the case that the two flows are assumed to be
smooth.  If one had any proof that only used the maximum principle, then 
 that proof should apply to any weak set flow (by definition of weak set flow). Thus, it seems  
   unlikely that assuming smoothness could result in any simplifications.

The organization of the paper is as follows.  Sections~\ref{conventions-section}
 and~\ref{prelim-section} describe conventions in the paper and 
prove some preliminary facts about weak set flows.
Section~\ref{main-section} gives the proof of the avoidance principle.
The proof uses various facts about the distance function, about certain harmonic
functions, and about separating Brakke flows.  
Those facts are proved in Sections~\ref{distance-section}, \ref{interpolation-section},
and~\ref{separating-section}.

To understand the main ideas of the proof, it  may be helpful to focus
on the case when the ambient space is Euclidean.  In that case, it is not necessary to use the exponential map.  Where the proof uses the exponential map
at a point $p$, in the Euclidean case one simply translates by $-p$ (and, in some cases, also rotates
and dilates).

The definition and basic properties of weak set flows may be found in~\cite{hershkovits-w-avoid}.
Sections 6 and 7 of \cite{ilmanen-elliptic} give a concise introduction to Brakke flows.

\section{Conventions}\label{conventions-section}

If $t\in [0,\infty)\mapsto M(t)$ is an integral 
Brakke flow in the Riemannian manifold $N$, we let $\MMM$ denote
its spacetime support:
\[
   \MMM:= \overline{ \{ (x,t): t\in [0,\infty), \, x\in \spt M(t)\} },
\]
and we let
\[
   \MMM(t) = \{x: (x,t)\in \MMM\}.
\]
Thus $t\mapsto \MMM(t)$ is the weak set flow associated to the Brakke flow $M(\cdot)$.

If $p$ and $q$ are points in a connected Riemannian manifold $N$ and if $X$ and $Y$ are subsets of $N$,
we let $d(p,q)$ be the geodesic distance from $p$ to $q$, and we let
\begin{gather*}
d(p,X) = d(X,p) = \inf_{x\in X} d(p,x), \\
d(X,Y) = \inf_{x\in X, y\in Y} d(x,y).
\end{gather*}

\section{Preliminaries}\label{prelim-section}

Throughout the paper, $N$ is a complete, connected Riemannian $(m+1)$-manifold (without boundary)
such that $\inj(N)>0$ (where $\inj(N)$ is the injectivity radius of $N$), and such that
\[
  \sup_N |\nabla^k \Riem|  <\infty
\]
for each nonnegative integer $k$.

\begin{lemma}[Finite Speed Lemma]\label{speed-lemma}
For $r>0$, there is an $h=h(r,\Lambda,m)$ with the following property.
If $S\subset N$, if $t\in [0,\infty)\mapsto X(t)$ is a weak set flow, and if
\[
    r < R:=d(X(0),S),
\]
then
\[
    d(X(t),S) \ge R - ht  \qquad\text{for $0\le t \le (R-r)/h$}.
\]
\end{lemma}

See~\cite{hershkovits-w-avoid}*{Theorem~5}.

The Finite Speed Lemma~\ref{speed-lemma}
implies control on how fast a weak set flow can move away from its initial set:

\begin{corollary}\label{speed-corollary}
For every $\eps>0$, there is a $\delta>0$ with the following property.
If $t\in[0,\infty)\mapsto Z(t)$ is a weak set flow in $N$, then
\[
  \cup_{t\in [0,\delta]} Z(t) \subset \{p:   d(p,Z(0)) \le \eps\}.
\]
\end{corollary}

The following corollary gives a bound on the rate at which two
weak set flows can approach each other.

\begin{corollary}\label{double-speed-corollary}
If $t\in [t_0,\infty)\mapsto X(t), \, Y(t)$ are weak set flows with
\[
   d(X(t_0),Y(t_0))  > r > 0,
\]
then
\[
    d(X(t_0+t),Y(t_0+t)) \ge d(X(t_0), Y(t_0))  - 2 h t
\]
for $t\le (R-r)/h$, where $h$ is as in Lemma~\ref{speed-lemma}.
\end{corollary}

\begin{proof}
It suffices to prove it for $t_0=0$.  Let
\[
   R = \frac12 d(X(t_0), Y(t_0)),
\]
and let $\tilde X = \{p: d(p,X(0))\ge R\}$ and $\tilde Y= \{p:d(p,Y(0)) \ge R\}$.
Then
\begin{gather*}
   d(X(t),\tilde X) \ge R - ht, \\
   d(Y(t), \tilde Y) \ge R - ht
\end{gather*}
for $t\in [0,(R-r)/h]$.  Thus
\[
d(X(t),Y(t)) \ge d(X(t),\tilde X) + d(\tilde Y,Y(t)) \ge 2R - 2ht.
\]
\end{proof}

\begin{proposition}\label{equivalence-proposition}
Suppose that $\lambda\in \RR$.
The following are equivalent.
\begin{enumerate}[\upshape (1)]
\item\label{equiv-1} If $t\in [t_0,\infty)\mapsto X(t),Y(t)$ are weak set flows in $N$,
then there is an $\eps>0$ such that 
\begin{equation*}
   e^{-\lambda (t-t_0)} d(X(t),Y(t)) \ge d(X(t_0), Y(t_0))
\end{equation*}
for all $t\in [t_0,t_0+\eps]$.
\item\label{equiv-2} If $t\in [t_0,\infty)\mapsto X(t),Y(t)$ are weak set flows that in $N$, then
\[
   e^{-\lambda (t-t_0)} d(X(t),Y(t)) \ge d(X(t_0),Y(t_0))
\]
for all $t\in [t_0,\infty)$.
\item\label{equiv-3} If $t\in [0,\infty)\mapsto X(t),Y(t)$ are weak set flows, then 
\[
   t\in [0,\infty)\mapsto e^{-\lambda t} d(X(t),Y(t))
\]
is an increasing function of $t$.
\end{enumerate}
Furthermore, if 
\[
   t\in [0,\infty) \mapsto e^{-\lambda t} d(X(t),Y(t))
\]
is an increasing function of $t$ for each $\lambda<\Lambda$, then it is also
an increasing function of $t$ for $\lambda=\Lambda$.
\end{proposition}

\begin{proof}
Suppose that~\eqref{equiv-1} holds.  In~\eqref{equiv-2}, let $\Tt$ be the set of $T\in [t_0,\infty)$ such that the 
 inequality in~\eqref{equiv-2}
 holds for all $t\in [t_0,T]$.
Then $\Tt$ is closed (by Corollary~\ref{double-speed-corollary}),
  and it is nonempty since $t_0\in \Tt$.
  By~\eqref{equiv-1}, $\Tt$ is relatively open in $[t_0,\infty)$.
Thus $\Tt=[t_0,\infty)$. 
Hence~\eqref{equiv-1} implies~\eqref{equiv-2}.

Trivially, \eqref{equiv-2} implies~\eqref{equiv-3} and~\eqref{equiv-3} implies~\eqref{equiv-1}. 
 The ``furthermore'' assertion is also trivially true.
\end{proof}

The proof of the Avoidance Theorem uses the following version of the maximum principle 
for weak set flows:

\begin{proposition}\label{key-fact}
Suppose that $\lambda\in \RR$,  
 that $W$ is a smooth (not necessarily complete)
Riemannian manifold, and 
    that $X(\cdot)$ and $Y(\cdot)$ are weak set flows in $W$. 
  Suppose that $T>0$ and that
\begin{align*}
   e^{-\lambda t} d(X(t),Y(t)) &> \eta \quad\text{for $0\le t<T$}, \\
   e^{-\lambda T} d(X(T),Y(T)) &= \eta.
\end{align*}
Suppose also that there is a geodesic $\Gamma$ of length $d(X(T),Y(T))$
from a point $x\in X(T)$ to a point $y\in Y(T)$, 
and that $(y,T)$ is a regular point of the flow $Y(\cdot)$.
Then there are points in $W$  (indeed, points on $\Gamma$) where the Ricci curvature is $\le \lambda$.
\end{proposition}

See~\cite{hershkovits-w-avoid}*{Lemma~10}.
In the statement of~\cite{hershkovits-w-avoid}*{Lemma~10},
the flow $K(\cdot)$ in the lemma is assumed to be smooth everywhere. But
the proof only requires $K(\cdot)$ to be smooth in a spacetime neighborhood of the point $(p,b)$
in that lemma.  
(Indeed, by replacing $W$ by a suitable open subset $W'$,
$[0,T]$ by a suitable subinteral $[t_0,T]$, and $X(\cdot)$ and $Y(\cdot)$ by
$X(\cdot)\cap W'$ and $Y(\cdot)\cap W'$, one can reduce the general case to the case
when $Y(\cdot)$ is smooth everywhere.)

\begin{corollary}\label{key-corollary}
Suppose $W$ is a smooth, complete Riemannian manifold with Ricci curvature bounded
below by $\Lambda$.    If $t\mapsto X(t)$ is a weak set flow in $W$,
if $\lambda<\Lambda$, and if $c\ge 0$, then
\[
 t\mapsto \{p: d(p,X(t)) \le ce^{\lambda t} \},
\]
is also a weak set flow.
\end{corollary}

Of course, since a limit of weak set flows is a weak set flow, Corollary~\ref{key-corollary} 
 also holds for $\lambda=\Lambda$.

\begin{comment}
To understand Proposition~\ref{key-fact}, it is instructive to see the proof  in the case when $N$ is
 $\RR^{m+1}$.
By translation, we can assume that $y=0$.  Let $X'(t)=X(t)-x$.  Then $X'(\cdot)$ is also a weak set flow.
If $\lambda \le 0$, then $d(X'(t),Y(t))>0$ for $t<T$ and $0\in X'(T)\cap Y(T)$.
Thus $X'(\cdot)$ and $Y(\cdot)$ first bump into each other at the origin at time $T$, which is impossible
(by definition of weak set flow) since $(0,T)$ is a regular point of the flow $Y(\cdot)$.  Thus $\lambda<0$,
so Proposition~\ref{key-fact} is proved in this case.
\end{comment}

\section{The Avoidance Theorem}\label{main-section}

\begin{theorem}[Avoidance Theorem]\label{avoidance-theorem}
Suppose that $N$ is a complete, connected Riemannian $(m+1)$-manifold (without boundary)
such that $\inj(N)>0$ and such that
\[
   \sup_N |\nabla^k\Riem| < \infty
\]
for each nonnegative integer $k$.  Let $\Lambda$ be a lower bound for the Ricci curvature of $N$.
Suppose that $t\in [0,\infty)\mapsto X(t), \, Y(t)$ are weak set flows in $N$.
Then 
\[
   e^{-\Lambda t} d(X(t),Y(t))
\]
is an increasing function of $t\in [0,\infty)$.
\end{theorem}

\begin{proof}

By Proposition~\ref{equivalence-proposition}, it suffices to prove that if $\lambda<\Lambda$, and if 
\[
   R = \frac12 d(X(0),Y(0)),
\]
then there is an $\eps>0$ such that 
\begin{equation}\label{to-show}
   \text{ $e^{-\lambda t} d(X(t),Y(t)) \ge 2R$ for $t\in [0,\eps]$}.
\end{equation}
We may assume that $R>0$, as~\eqref{to-show} is trivially true for $R=0$.
We may also assume that $\Lambda\le 0$, as otherwise $N$ is compact,
and the avoidance principle is already known in that case.  See~\cite{hershkovits-w-avoid}.

\begin{case}\label{case-1} $d(X(0),Y(0)) \le  r:= \frac12 \inj(N)$.
\end{case}

Let $X=X(0)$ and $Y=Y(0)$.
Let $K$ be the set of points $p$ such that $d(p,X\cup Y)\ge R$.
Let $S$ be the set of points $p$ such that $d(p,X)=d(p,Y)=R$.
Thus $S\subset \partial K$.

We will use the following interpolation theorem.

\begin{theorem}[Interpolation Theorem]\label{restated}
There is a closed region $\Omega$ in $N$ with the following properties:
\begin{enumerate}
\item $\Omega$ is a $C^1$ manifold-with-boundary.
\item $\Omega$ contains $\{p: d(p,X)\le R\}$ and is disjoint from $\{p: d(p,Y)<R\}$.
\end{enumerate}
Furthermore,  $\Sigma:=\partial\Omega$ is uniformly $C^1$ near $\partial K$ in the following sense.
For every $\eps>0$, there is a $\delta>0$ such that if $p,q\in \Sigma$, if $d(p,\partial K)<\delta$,
 and  if $d(p,q)<\delta$, then 
\[
   d(\nu(p),\nu(q)) \le  \eps.
\]
\end{theorem}

Here, $\nu$ is the unit normal to $\Sigma$ that points out from $\Omega$.
The meaning of $d(\nu(p),\nu(q))$ is given by:

\begin{definition}\label{tangent-distance-definition}
If $\uu$ and $\vv$ are tangent vectors to $N$, we let $d(\uu,\vv)$
be the supremum of 
$
   |\uu\cdot \ww - \vv\cdot \ww|
$
among $C^1$ vectorfields  $\ww$ on $N$ such that $|\ww|\le 1$ and 
$|\nabla \ww|\le 1$ at all points.
\end{definition}

The Interpolation Theorem is proved in \S\ref{interpolation-section}.

Let $t\in [0,\infty)\mapsto M(t)$ be a separating Brakke flow such that $M(0)$ is the Radon
measure associated to $\Sigma$.   See \S\ref{separating-section} for the definition of ``separating flow"
and for a proof that the flow $M(\cdot)$ exists.

If $(p,t)$ is a regular point of $\MMM$, we let $a(\MMM,p,t)$ be the norm of the second fundamental 
form of $\MMM(t)$ at $p$.  If $(p,t)$ is a singular point of $\MMM$, we 
let $a(\MMM,p,t)=\infty$.  We let
\[
   A(\MMM,p,t) = \sup\{ a(\MMM,p',t'):  d(p',p)\le |t|^{1/2}, \,  t/2 \le t' \le 2t\}.
\]

\begin{claim}\label{claim-one}
There is an $\eps>0$ with the following property.
If $0<T\le \eps$, if $p\in \MMM(T)$, and if $d(p,K^c)\le T^{1/2}$, then
\[
  T^{1/2}   A(\MMM,p,T) < 1.
\]
\end{claim}

\begin{proof}[Proof of Claim~\ref{claim-one}]
Suppose not.  Then there exist $t_i\to 0$ and $p_i\in \MMM(t_i)$ such that 
\[
   d(p_i, K^c) \le t_i^{1/2}
\]
and such that
\[
  t_i^{1/2} A(\MMM,p_i,t_i) \ge 1.
\]
Let $L_i: \RR^{m+1} \to \Tan(N,p_i)$ be a linear isometry,
and let 
\begin{align*}
&\phi_i:  B_i = B^{m+1}(0, r / \sqrt{t_i}) \to N, \\
&\phi_i(v) = \exp_{p_i} ( L_i (t_i^{1/2} v)),
\end{align*}
where $r=\inj(N)/2$.  We give $B_i$ the Riemmannian metric $g_i$ such  that
a curve of $g_i$-length $s$ in $B_i$ is mapped (by $\phi_i$) to a curve of length $t_i^{1/2}s$ in $N$.

Let $\Sigma_i= \phi_i^{-1}(\Sigma)$.

Let $M_i$ be the Brakke flow in $B_i$ obtained by applying $\phi_i^{-1}$ (parabolically)
 to the Brakke flow $M$.

After passing to a subsequence,  $\Sigma_i$ converges to a closed subset $\Sigma'$ of $\RR^{m+1}$.

We claim that $\Sigma'$ is nonempty.  For if it were empty, then (after passing to a further subsequence)
the flows $\MMM_i$ would converge to a weak set flow $t\mapsto Z(t)$ in $\RR^{m+1}$
with $0\in Z(1)$ and with $Z(0)=\emptyset$, which is impossible (by
  Lemma~\ref{speed-lemma}, for example).
  
Thus $\Sigma'$ is nonempty.
Thus if $\tilde p_i$ is a point in $\Sigma$ closest to $p_i$, then
\[
   d(\tilde p_i, p_i) \le kt_i^{1/2}
\]
for some $k<\infty$ independent of $i$.
Hence
\[
  d(\tilde p_i, K^c) \le d(\tilde p_i,p_i) + d(p_i,K^c) \to 0.
\]
  Consequently, if $\tilde q_i\in \Sigma$ and $d(\tilde q_i,\tilde p_i)\to 0$,
then
\begin{equation}\label{same}
     d(\nu(\tilde q_i), \nu(\tilde p_i)) \to 0
\end{equation}
by Theorem~\ref{restated}.

It follows that $\Sigma'$ is a multiplicity-$1$ plane.
(To see that there is only one plane, counting multiplicity,
 note that $\Sigma'$ is a union of oriented planes having the same normal $\nu$, by~\eqref{same}.
In particular, all the planes in $\Sigma'$ are parallel.  If there were more than one
plane, counting multiplicity, then some of the planes would have normal $\nu$ and others
would have normal $-\nu$, since $\Sigma$ bounds a region $\Omega$.)

After passing to a further subsequence, the $M_i$ converge to an integral
Brakke flow $M'$, where $M'(0)$ is given by the plane $\Sigma'$ with
multiplicity $1$.  
Since the flow is separating, it follows (see Lemma~\ref{separating-lemma})
that $M'(t)$ is the plane $\Sigma'$ with multiplicity $1$ for all $t\ge 0$.

By local regularity~\cite{white-local}, the convergence $\MMM_i\to \MMM'$ is smooth on compact
subsets of $\RR^{m+1}\times (0,\infty)$.  Thus
\[
   A(\MMM_i,0,1) \to 0,
\]
But, by hypothesis, $A(\MMM_i,0,1)\ge 1$ for all $i$.
  The contradiction proves Claim~\ref{claim-one}.
\end{proof}

\begin{claim}\label{claim-two}
There is a $\delta>0$ such that 
\[
    e^{-\lambda t} d(X(t),\MMM(t)) \ge R
\]
and
\[
   e^{-\lambda t} d(Y(t),\MMM(t)) \ge R
\]
for $0\le t\le \delta$.
\end{claim}

\begin{proof}
It suffices to prove it for $X(\cdot)$.
Recall that there exist $C<\infty$ and $\eta<0$ such that
\begin{equation}\label{C-eta}
    \text{$d(X(t), K) \ge R - Ct$ for $0\le t\le \eta$.}
\end{equation}
(See Lemma~\ref{speed-lemma}.)
Choose $\delta>0$ small enough that
\begin{align*}
\delta&\le \eps, \\
\delta&\le \eta, \\
C\delta^{1/2} &\le \frac12
\end{align*}
where $\eps$ is as in Claim~\ref{claim-one} and $\eta$ and $C$ are as in~\eqref{C-eta}.
Thus if $0\le t \le \delta$, then
\begin{equation}\label{Ct-small}
  Ct \le C \delta^{1/2} t^{1/2}  \le  \frac12 t^{1/2}.
\end{equation}
Let 
\begin{align*}
&f: [0,\delta] \to \RR, \\
&f(T) = \inf_{0\le t\le T} e^{-\lambda t} d(X(t),\MMM(t)).
\end{align*}
Suppose the claim is false.  Then $f(T)<R$ for some $T\in (0,\delta]$.
 By Corollary~\ref{double-speed-corollary}, $f$
is continuous.
Thus, by replacing $T$ by a smaller $T>0$, we can
assume that 
\[
   0 < f(T) < R.
\]
We can also assume that $f(t)>f(T)$ for $t<T$; otherwise replace $T$
by the smallest $t$ such that $f(t)=f(T)$.
Thus
\[
   \eta:=f(T) < f(t) \qquad\text{for $t<T$}.
\]
In particular, $e^{-\lambda T} d(X(T),\MMM(T)) =\eta$,
so
\begin{equation}\label{close-enough}
d(X(T),\MMM(T)) = e^{\lambda T}\eta < \eta < R \le \frac12r
\end{equation}
since we are assuming that $\lambda<0$ and (in Case~\ref{case-1}) that $R\le \frac12r$.

Choose $p_i\in \MMM(T)$ and $q_i\in X(T)$ so that
\[
   d(p_i,q_i) \to d(\MMM(T),X(T)) = e^{\lambda T}\eta.
\]

By passing to a subsequence, we can assume that $d(K^c,q_i)$ converges to a limit in $[0,\infty]$.
We wish to show that
\begin{equation}\label{the-inequality}
  \lim_i d(q_i, K^c) \le \frac12 T^{1/2}.
\end{equation}
We may assume that $q_i$ is in the interior of $K$ for all sufficiently large $i$, as otherwise
the inequality~\eqref{the-inequality} is trivially true. 
For such $i$,
\begin{align*}
d(p_i,q_i)
&\ge 
d(p_i, \partial K) + d(\partial K,q_i) \\
&\ge
R- CT + d(K^c,q_i).
\end{align*}
Thus
\[
e^{\lambda T}\eta \ge R - CT + \lim_i d(K^c,q_i).
\]
Now $e^{\lambda T} \le 1$ (since we are assuming that $\Lambda\le 0$)
 and $\eta<R$, so
\[
   \lim_i d(q_i,K^c) \le CT < \frac12 T^{1/2}
\]
by~\eqref{Ct-small}.
This completes the proof that the inequality~\eqref{the-inequality} holds.

By  Claim~\ref{claim-one} and~\eqref{the-inequality},
\begin{equation}\label{smooth-point}
 |T|^{1/2} A(\MMM_i,q_i,T) < 1
\end{equation} 
for all sufficiently large $i$.

Now let $L_i: \RR^{m+1} \to \Tan(N,p_i)$ be a a linear isometry,
and let
\begin{align*}
&F_i: B=B^{m+1}(0,r) \to N, \\
&F_i = \exp_{p_i}\circ L_i.
\end{align*}
Let $g_i$ be the pull-back by $F_i$ of the metric on $N$, so that
$F_i$ maps $(B,g_i)$ isometrically onto $B(p_i,r)$.

For large $i$, $d(p_i,q_i) < r$ (by~\eqref{close-enough}),
so there is a $\tilde q_i\in B$ with $F_i(\tilde q_i) = q_i$.
Let 
\begin{align*}
\tilde X_i(t) &= F_i^{-1}(X(t)), \\
\tilde \MMM_i(t) &= F_i^{-1}(\MMM(t)).
\end{align*}
Thus $\tilde X_i(\cdot)$ and $\tilde \MMM_i(\cdot)$ are weak set flows in $B$ with respect to $g_i$.

After passing to a subsequence, the $g_i$ converge to a smooth Riemannian metric $\tilde g$,
  $\tilde X_i(\cdot)$ and $\tilde \MMM_i(\cdot)$ 
converge to weak set flows $\tilde X(\cdot)$ and $\tilde \MMM(\cdot)$ in $(B,\tilde g)$,
 and $\tilde q_i$ converges to point $\tilde q$ with
 \begin{equation}\label{attained}
    d(0,\tilde q) = |\tilde q| = d(X(T),\MMM(T))= e^{\lambda T}\eta.
\end{equation}

For $0\le t<T$, 
\begin{equation}\label{not-attained}
  e^{-\lambda t} d(\tilde X(t), \tilde \MMM(t)) \ge  e^{-\lambda t}d(X(t), \MMM(t)) > \eta.
\end{equation}

By~\eqref{smooth-point},
 the spacetime point $(0,T)$ is a regular point of the flow $\tilde \MMM$.
Thus by~\eqref{attained}, \eqref{not-attained}, and~Proposition~\ref{key-fact},
    $(B,\tilde g)$ has points of Ricci curvature $\le\lambda$,
contrary to our choice of $\lambda<\Lambda$.  
This completes the proof of Claim~\ref{claim-two}.
\end{proof}

\begin{claim}\label{claim-three}
There is an $\eps>0$ with the following property:
\begin{gather*}
\cup_{t\in [0,\eps]} X(t)  \subset \{p: d(p,X) < R/8\}, \\
\cup_{t\in [0,\eps]} \MMM(t)  \subset \{p: d(p,\MMM(0)) < R/8\}, \\
\cup_{t\in [0,\eps]} Y(t)  \subset \{p: d(p,Y) < R/8\}.
\end{gather*}
\end{claim}

Claim~\ref{claim-three} follows immediately from Corollary~\ref{speed-corollary}.

\begin{claim}\label{claim-four}
If $0<T\le \eps$, where $\eps$ is as in Claim~\ref{claim-three},
then
\[
    d(X(T),Y(T)) \ge d(X(T), \MMM(T)) + d(\MMM(T),Y(T)).
\]
\end{claim}

\begin{proof}[Proof of Claim~\ref{claim-four}]
Let $p\in X(T)$ and $q\in Y(T)$.  
Let $\alpha:[0,T]\to N$ be a geodesic such that
\begin{gather*}
\alpha(0)\in X, \\
\alpha(T)=p, \\
\Length(\alpha) = d(p,X).
\end{gather*}
Let $\beta:[T,2T]\to N$ be a shortest geodesic from $p$ to $q$.
Let $\gamma:[2T,3T]\to N$ be a geodesic such that
\begin{gather*}
\gamma(2T)=q, \\
\gamma(3T)\in Y, \\
\Length(\gamma)= d(q, Y).
\end{gather*}
Consider the path
\[
\mu: s\in [0,3T]
\mapsto
\begin{cases}
(\alpha(s),s)  &\text{if $s\in [0,T]$}, \\
(\beta(s),T)   &\text{if $s\in [T,2T]$}, \\
(\gamma(s), 3T-s)  &\text{if $s\in [2T,3T]$}.
\end{cases}
\]
Thus $\mu$ is a path in $N\times [0,\infty)$ from $(p,0)$ to $(q,0)$.
Since $M(\cdot)$ is separating, there must be an $s\in [0,3T]$ such
that 
\[
    \mu(s) \in \MMM.
\]
By Claim~\ref{claim-three}, $s$ cannot be in $[0,T]$ or in $[2T,3T]$.
Thus $s \in (T,2T)$, so
\[
   \beta(s) \in \MMM(T).
\]
Therefore,
\begin{align*}
d(X(T),Y(T)) 
&\ge
d(p,q)  \\
&=
d(p,\beta(s)) + d(\beta(s),q) \\
&\ge
d(p,\MMM(T)) + d(\MMM(T),q).
\end{align*}
Taking the infimum over $p\in X(T)$ and $q\in Y(T)$ gives
\[
 d(X(T),Y(T)) \ge d(X(T),\MMM(T)) + d(\MMM(T),Y(T)).
\]
This completes the proof of Claim~\ref{claim-four}.
\end{proof}

Now let $\tilde \eps$ be the minimum of 
the $\delta$ in Claim~\ref{claim-two} and the $\eps$ in Claim~\ref{claim-four}.
Then for $0\le t\le \tilde \eps$,
\begin{align*}
d(X(t),Y(t)) 
&\ge
\d(X(t),\MMM(t)) + d(\MMM(t),Y(t)) 
\\
&\ge
2R e^{\lambda t}
\end{align*}
This completes the proof of~\ref{to-show} in Case~\ref{case-1} of the Avoidance Theorem.

\begin{case}\label{case-2}   $d(X(0),Y(0)) > r = \frac12 \inj(N)$.
\end{case}

Let 
\[
    c= d(X(0),Y(0)) - r,
\]
and let
\[
\tilde X(t)= \{p: d(p,X(t)) \le ce^{\lambda t} \}
\]
for $t\ge 0$.   
Then $\tilde X(\cdot)$ is a weak set flow (by Corollary~\ref{key-corollary}),
and
\[
  d(\tilde X(0), Y(0)) = r.
\]
Thus by Case~\ref{case-1}, 
\begin{equation}\label{tilde-version}
    e^{-t}d(\tilde X(t), Y(t)) \ge d(\tilde X(0), Y(0))  
\end{equation}
for $t\in [0,\eps]$, for some $\eps>0$.

If $d(\tilde X(t),Y(t))>0$, then 
\[
   d(X(t),Y(t)) = d(\tilde X(t), Y(t)) + ce^{\lambda t}.
\]
Thus~\eqref{tilde-version} can be written as
\[
   e^{-\lambda t} d(X(t),Y(t)) + c \ge d(X(0),Y(0)) + c
\]
or
\[
   e^{-\lambda t}d(X(t),Y(t)) \ge d(X(0),Y(0))
\]
for $t\in [0,\eps]$.  This completes the proof of~\eqref{to-show} in Case~\ref{case-2}.
\end{proof}

\section{Properties of the Distance Function}\label{distance-section}
     
This section proves properties of the distance function that are used in the proof
of the Interpolation Theorem~\ref{interpolation-section} in Section~\ref{interpolation-section}.

Throughout this section,  $X$ and $Y$ are closed sets in $N$ such that $d(X,Y)>0$.
Let
\begin{align*}
R &= \frac12 d(X,Y), \\
K &= \{p: d(p,X\cup Y)\ge R\}, \\
S &= \{p: d(p,X)=d(p,Y) = R\}.
\end{align*}
We abbreviate $d(p,X)$ by $x(p)$ and $d(p,Y)$ by $y(p)$.

First, we prove a useful fact about geodesics in $N$.

\begin{proposition}\label{offset-proposition}
For every $\eps>0$ and $\sigma>0$, there is a $\delta>0$ with the following property.
If $\alpha:[-a,0]\to N$ and $\beta:[0,b]\to N$ are unit-speed
geodesics with $a,b\ge \sigma$, and if 
\begin{equation}\label{conditions}
\begin{gathered}
d(\alpha(0),\beta(0)) \le \delta, \\
d(\alpha(-a),\beta(b)) \ge a +b - \delta,
\end{gathered}
\end{equation}
then $d(\alpha'(0),\beta'(0))<\eps$.
\end{proposition}

(See Definition~\ref{tangent-distance-definition} for the meaning of $d(\alpha'(0),\beta'(0))$.)

\begin{proof}
If the theorem holds for one $\sigma$, then it holds for any large $\sigma$.
Thus we may assume that $\sigma< r:=\frac12\inj(N)$.

Suppose that $\alpha$ and $\beta$ are unit-speed geodesics that satisfy~\eqref{conditions}.
Then
\begin{align*}
a+b-\delta
&\le
d(\alpha(-a),\beta_i(b)) \\
&\le 
d(\alpha_i(-a), \alpha(-\sigma)) + d(\alpha(-\sigma),\beta(\sigma)) + d(\beta(\sigma),\beta(b)) \\
&\le 
a-\sigma + d(\alpha(-\sigma),\beta(\sigma))  + b -\sigma,
\end{align*}
So
\[
2\sigma -\delta \le d(\alpha(-\sigma), \beta(\sigma)).
\] 
Thus if we replace $a$ and $b$ by $\sigma$ and $\alpha$ and $\beta$
by their restrictions to $[-\sigma,0]$ and $[0,\sigma]$, the conditions~\eqref{conditions}
still hold.  Consequently, it suffices to prove the proposition when $a=b=\sigma$.

Thus it suffices to show that if $\alpha_i:[-\sigma,0]\to N$ and $\beta_i:[0,\sigma] \to N$
are unit-speed geodesics, and if
\begin{gather*}
d(\alpha_i(0), \beta_i(0)) \to  0\\
\liminf_i d(\alpha_i(-\sigma), \beta_i(\sigma)) \ge 2\sigma,
\end{gather*}
then 
\[
  d(\alpha_i'(0), \beta_i'(0)) \to 0.
\]

Let $p_i=\alpha_i(0)$, 
let $L_i: \RR^{m+1}\to \Tan(N,p_i)$ be a linear isometry,
and let
\begin{align*}
&F_i:\RR^{m+1}\to N, \\
&F_i(v) = \exp_{p_i}(L_iv).
\end{align*}
Let $g_i$ be the metric on $B=B^{m+1}(0,r)$ obtained
by pulling-back the metric on $N$ by $F_i$.
Define $\talpha_i:[-\sigma,0] \to B$ and $\tbeta_i: [0,\sigma] \to B$
by
\begin{gather*}
\talpha_i(s) = F_i^{-1}\alpha_i(s), \\
\tbeta_i(s) = F_i^{-1}\beta_i(s)
\end{gather*}
Then $\talpha_i$ and $\tbeta_i$ are unit-speed 
geodesics with respect to $g_i$.

By passing to a subsequence, we can assume that $g_i$ converges smoothly
to a metric $g$, and that $\talpha_i$ and $\tbeta_i$ converge smoothly
to unit-speed $g$-geodesics $\alpha:[-\sigma,0]\to B$ and $\beta:[0,\sigma] \to B$.
Note that $\alpha(0)=\beta(0)$ and that
\[
  d_g(\alpha(-\sigma),\beta(\sigma)) \ge 2\sigma.
\]
Consider the map
\[
\gamma: s\in [-\sigma,\sigma]
\mapsto
\begin{cases}
\alpha(s) &\text{if $s\in [-\sigma,0]$}, \\
\beta(s) &\text{if $s\in [0,\sigma]$}.
\end{cases}
\]
Then
\[
  2\sigma \le d_g(\gamma(-\sigma),\gamma(\sigma)) \le \Length(\gamma)=2\sigma.
\]
Thus $\gamma$ is a unit-speed geodesic, so it is smooth, and thus
\[
\alpha'(0)=\beta'(0).
\]
\end{proof}

For the following definition, recall that $x(p)=d(p,X)$ and $y(p)=d(p,Y)$.

\begin{definition}
For $p\in N\setminus(X\cup Y)$, let $\Xx(p)$ be the set of unit vectors $\uu\in \Tan(N,p)$
such that 
\[
   \exp_p( - x(p) \uu) \in X,
\]
and let $\Yy(p)$ be the set of unit vectors $\vv\in \Tan(N,p)$ such that 
\[
  \exp_p(y(p)\vv) \in Y.
\]
\end{definition}

Note that for each $x\in (X\cup Y)$, $\Xx(p)$ and $\Yy(p)$ are nonempty.
Note also that if $\alpha:[-a,0]\to X$  is a unit-speed geodesic
with with $\alpha(0)=p$ and with length $a=d(X,p)=x(p)$,
 then $\alpha'(0)\in \Xx(p)$.  Conversely, if $\uu\in \Xx(p)$,
then 
\[
   s \in [-x(p),0] \mapsto \exp_p(s\uu)
\]
is such a geodesic.  Of course, the analogous statements hold for $\vv\in \Yy(p)$.

\begin{remark}\label{exterior-sphere-remark}
Suppose that $\uu\in \Xx(p)$ and let $\alpha(s)=\exp_p(s\uu)$ for $-x(p)\le s \le 0$.
Note that $x(\cdot)\le x(p)$ on ball $B(\alpha(s),-s)$.
  Likewise, if $\vv\in \Yy(p)$ and if $\beta(s)=\exp_p(s\vv)$ for $0\le s\le y(p)$, then
  $y(\cdot)\le s$ on the ball $B(\beta(s),s)$.
\end{remark}

\begin{theorem}\label{u-v-theorem}
For every $\eps>0$, there is an $\eta>0$ with the following property.
Suppose that
\begin{gather*}
\uu\in \Xx(p), \\
\vv\in \Yy(q), \\
|x(p) - R| \le \frac12\eta, \\
|y(q) < R| \le  \frac12\eta, \\
d(p,q) \le \frac12\eta.
\end{gather*}
Then 
$
 d(\uu,\vv) < \eps.
$
\end{theorem}

\begin{proof}
Let $\sigma= R/2$, and let $\delta$ be as in Proposition~\ref{offset-proposition}.
Let $\eta$ be the smaller of $\delta$ and $R/2$.

Let $a=x(p)$, $b=y(q)$, and let
\begin{align*}
\alpha(s) &= \exp_p(s\uu)    \qquad (s\in [-a,0]), \\
\beta(s) &= \exp_q(s\vv)  \qquad (s\in [0,b]).
\end{align*}
Then
\begin{align*}
d(\alpha(-a),\beta(b))
&\ge
d(X,Y) \\
&= 
2R \\
&\ge
x(p) + y(p) -\delta \\
&=
a + b -\delta.
\end{align*}
Thus $d(\uu,\vv)<\eps$ by choice of $\eta$; see Proposition~\ref{offset-proposition}.
\end{proof}

\begin{corollary}\label{u-v-corollary}
Let $\eps$ and $\eta$ be as in Theorem~\ref{u-v-theorem}.
Suppose that 
\begin{gather*}
d(p_1,p_2) \le \eta/2, \\
|x(p_i) -R| \le \eta/2, \\
\uu_i\in \Xx(p_i),\\
\vv_i\in \Yy(p_i).
\end{gather*}
for $i=1,2$.
Then
$
  d(\uu_1, \uu_2) < 2\eps.
$
and $d(\vv_1,\vv_2)< 2\eps$.
\end{corollary}

\begin{proof}
By Theorem~\ref{u-v-theorem},
\[
  d(\uu_i,\vv_1) < \eps
\]
for $i=1,2$. Thus
$
  d(\uu_1,\uu_2) < 2\eps.
$
Similarly, $d(\vv_1,\vv_2)\le 2\eps$.
\end{proof}

\begin{corollary}\label{unique-u-corollary}
Let $p\in S$.  Then there is a unique vector $\uu=\uu(p)$ in $\Xx(p)$,
and $\uu(p)$ is also the unique vector in $\Yy(p)$.
Furthermore, $\uu(p)$ depends continuously on $p\in S$.
\end{corollary}

The following theorem describes the function $x(\cdot)$ near
 a point $p$ where $x(p)$ and $y(p)$ are both very close to $R$.
 Roughly speaking, it says that if $\uu\in \Xx(p)$ at such a point $p$,
 then $u-u(p)$ behaves (near $p$) like a $C^1$ function
 whose gradient at $p$ is $\uu$.

\begin{theorem}\label{local-theorem}
Suppose that $p_i\in N$, that $\uu_i\in \Xx(p_i)$, and that
\[
  |x(p_i)-R| + |y(p_i) -R| \to 0.
\]
Let $L_i: \RR^{m+1}\to \Tan(N,p_i)$ be a linear isometry
such that 
\[
|L\ee_1 - \uu_i| \to 0.
\]
Let $r_i\to 0$, and define
\[
F_i: B_i = B^{m+1}(0,r/r_i) \RR^{m+1} \to N
\]
by
\[
 F_i(v) = \exp_{p_i} L (r_iv),
\]
where $r=\frac12 \inj(N)$.

Let $g_i$ be the metric on $B_i$ such that $F_i$ maps a curve of $g_i$-length $s$
to a curve of length $r_is$

Define $x_i: B_i\to \RR$ by
\[
x_i(v)
=
\frac{x(F_i(v)) - x(F_i(0))}{r_i} = \frac{x(F_i(v)) - x(p_i)}{r_i}.
\]
Then $x_i(\cdot)$ converges uniformly on compact sets to the function
\begin{align*}
\tilde x:\RR^{m+1}\to \RR, \\
\tilde x(v) = v\cdot \ee_1.
\end{align*}
\end{theorem}

\begin{proof}
Note that $g_i$ converges smoothly to the Euclidean metric.
The function $x_i$ is $1$-Lipschitz with respect to $g_i$
and $x_i(0)=0$.  Thus, after passing to a subsequence, the functions $x_i$ converge uniformly
on compact sets to a $1$-Lipschitz function $\tilde x$.

Consider a $q\in \RR^{m+1}$.  If $q\in B_i$, let $\uu_i\in \Xx(F_i(q))$,
and let $\ww_i$ be the unit vector such that
\[
   DF_i(q)\ww_i = r_i \uu_i.
\]
Then
\[
   x_i(\exp_{q,i}{s\ww_i}) = x_i(q)  + s \qquad\text{for $s\le 0$},
\]
where $\exp_{q,i}$ is the exponential map at $q$ with respect to the metric $g_i$.
Now $\ww_i\to \ee_1$ (by Corollary~\ref{u-v-corollary}), so
\[
  \exp_{q,i}(s\ww_i) \to \exp_q(s\ee_1) = q + s\ee_1,
\]
and thus
\begin{equation}\label{shifty}
  \tilde x(q + s\ee_1) = \tilde x(q) + s 
\end{equation}
for all $q$ and all $s\le 0$.  Replacing $q$ by $q-s\ee_1$, we see that
\[
  \tilde x(q) = \tilde x(q-s\ee_1) + s,
\]
or
\[
 \tilde x(q-s\ee_1) = \tilde x(q) - s.
\]
Thus~\eqref{shifty} holds for all $q$ and all $s\in\RR$.
Since $\tilde x(0)=0$ and since $\tilde x$ is $1$-Lipschitz, 
it follows that 
\[
  \tilde x(v)\equiv v\cdot \ee_1.
\]
(If this is not clear, see Lemma~\ref{tight-lemma} below.)
\end{proof}

\begin{remark}\label{local-remark}
In Theorem~\ref{local-theorem}, let $\vv_i\in \Yy(p_i)$.
By Theorem~\ref{u-v-theorem}, $|\vv_i-\uu_i| \to 0$.
Thus (reversing the roles of $X$ and $Y$),
we see from Theorem~\ref{local-theorem} that the functions
\[
y_i(v) := \frac{y(F_i(v)) - y(p_i)}{r_i}
\]
converge uniformly on compact sets to the function
\begin{align*}
&\tilde y: \RR^{m+1}\to \RR, \\
&\tilde y(v) = - v\cdot \ee_1.
\end{align*}
\end{remark}

The following lemma was used at the end of the proof of Theorem~\ref{local-theorem}.

\begin{lemma}\label{tight-lemma}
Suppose that $\uu$ is a unit vector in $\RR^{m+1}$ and
that $f: \RR^{m+1}\to \RR$ is a $1$-Lipschitz function such that
\[
   f(s\uu) = s
\]
for all $s\in \RR$.  Then $f(v)= v\cdot\uu$ for all $v$.
\end{lemma}

\begin{proof}
If $r>0$, then for $p$ in the ball $B((s-r)\uu,r)$,
\begin{align*}
f(p)
&\le f((s-r)\uu) + r  \\
&= s.
\end{align*}
Letting $r\to\infty$ shows that $f\le s$ on the halfspace $\{v:v\cdot\uu\le s\}$.
Thus $f(v)\le v\cdot\uu$ for all $v$.

Likewise (using the ball $B((s+r)\uu,r)$), we see that $f(v)\ge v\cdot\uu$
for all $v$.
\end{proof}

\section{The Interpolation Theorem}\label{interpolation-section}

In this section, we prove the Interpolation Theorem that was used
in the proof of the Avoidance Theorem~\ref{avoidance-theorem}.

For $0<\rho<R$, let
\begin{gather*}
  K(\rho)= \{p: d(p,X\cup Y) \ge \rho\}.
%  \partial_XK(\rho) = \{p: d(p,X) = \rho\}, \\
% \partial_YK(\rho) = \{p: d(p,Y) =\rho\}.
\end{gather*}
Let $h^\rho:K(\rho)\to [0,1]$ be a harmonic function with boundary values
\[
h^\rho(p)
=
\begin{cases}
0 &\text{if $x(p)=\rho$},  \\
1 &\text{if $y(p)=\rho$}.
\end{cases}
\]

Existence of $h^\rho$ can be proved in a variety ways, such as the Perron method
as described in~\cite{GT}*{\S 2.8}.
Note that that the boundary values can be prescribed because $K(\rho)$ satisfies
an exterior sphere condition; see Remark~\ref{exterior-sphere-remark}.

\begin{lemma}\label{harmonic-lemma}
For every $\eps>0$, there is a $\delta$ with the following property.
If
\begin{gather*}
\rho>R-\delta, \\
h^\rho(p)\in [1/3,2/3], \\
d(p,\partial K(\rho))<\delta,  \\
\uu\in \Xx(p),
\end{gather*}
then
\begin{enumerate}
\item $|x(p)-R| + |y(p)-R|<\eps$.
\item $\nabla h^\rho(p)\ne 0$.
\item $d(\ww(p),\uu) < \eps$, where 
\[
  \ww(p) = \frac{\nabla h^\rho(p)}{|\nabla^\rho h(p)|}.
\]
\end{enumerate}
\end{lemma}

\begin{proof}
Suppose that $0<\rho_i<R$, that $\rho_i\to R$,
that
 $p_i \in K(\rho_i)$, that $\uu_i\in \Xx(p_i)$, 
 and that 
\begin{gather*}
h(p_i)\in [1/3, 2/3], \\
d(p_i,\partial K(\rho_i)) \to  0.
\end{gather*}
It suffices to show that 
\[
\nabla^{\rho_i}h(p_i)\ne 0
\]
 for large $i$,
and that
\begin{gather*}
|x(p_i)-R| + |y(p_i) - R| \to 0, \\
d(\ww_i, \uu_i) \to 0,
\end{gather*}
where
\[
  \ww _i = \frac{\nabla h^{\rho_i}(p_i)}{|\nabla h^{\rho_i}(p_i)|}.
\]

Let
\[
   \eta_i = \dist(p_i, \partial K(\rho_i)).
\]
Thus 
$
 \eta_i \to 0.
$

Let $L_i:\RR^{m+1} \to \Tan(N,p_i)$ be a linear isometry
such that $L_i\ee_1=\uu_i$.

Let $B_i= B^{m+1}(0, r/\eta_i)$ (where $r=\frac12\inj(N)$),
and let
define $F_i:B_i\to N$ by
\[
F_i(v) = \exp_{p_i}(\eta_i v).
\]

We endow $B_i$ with the metric $g_i$ such that if $\Gamma$ is a curve in $B_i$ with $g_i$-length $s$,
then its image under $F_i$ is a curve of length $\eta_i s$.

Let $K_i= F_i^{-1}(K)$.

By Theorem~\ref{local-theorem} and Remark~\ref{local-remark},
 the $K_i$ converge (perhaps after passing to a subsequence) to a limit $K'$
of the form
\[
  K' = \{p:  a\le p\cdot\ee_1 \le b\},
\]
where $-\infty\le a\le b \le \infty$.
Since
\[
    \eta_i = d(p_i,K^c),
\]
we see that $d(0,(K')^c)=1$.   Thus $a\in [-\infty, -1]$ and $b\in [1,\infty]$,
and either $a=-1$ or $b=1$.

Note that the $g_i$-harmonic functions $h\circ F_i: B_i\to [0,1]$
converge, perhaps after passing to a further subsequence, to a harmonic
function $h': K' \to [0,1]$ such that
\begin{align*}
h'(p) &= 0  \qquad\text{if $p\cdot \ee_1=a$}, \\
h'(p) &=1   \qquad\text{if $p\cdot \ee_1=b$},
\end{align*}
and such that
\begin{equation}\label{fixed}
  h'(0)=c\in [1/3,2/3].
\end{equation}
If $b=\infty$, then $a=-1$, and thus $h'\equiv 0$,  contrary to~\eqref{fixed}.
Likewise, if $a=-\infty$, then $b=1$, and thus $h'\equiv 1$, contrary to~\eqref{fixed}.
Consequently, $a$ and $b$ are both finite, so
\[
  h'(p) = \frac{p\cdot\ee_1 - a}{b-a}.
\]
and therefore
\[
   \frac{\nabla h'}{|\nabla h'|} = \ee_1.
\]

The assertions of the theorem follow immediately.
\end{proof}

\begin{corollary}\label{harmonic-corollary}
For every $\eps>0$, there is a $\delta>0$ with the following property.
If 
\begin{gather*}
R-\delta < \rho < R, \\
p,q\in K(\rho), \\
d(p,q)<\delta, \\
d(p,\partial K(\rho)) < \delta, \\
h^\rho(p), h^\rho(q)  \in [1/3, 2/3],
\end{gather*}
Then 
\[
    d(\ww(p),\ww(q)) < \eps.
\]
\end{corollary}

\begin{proof}
Let $\uu_p\in \Xx(p)$ and $\uu_q\in \Xx(q)$.
By Lemma~\ref{harmonic-lemma}, we can, by choosing $\delta$ small, ensure that
\[
  d(\ww(p),\uu_p) + d(\ww(q),\uu_q) < \eps/2.
\]
By Corollary~\ref{u-v-corollary}, we can choose $\delta$ small enough that
\[
   d(\uu_p, \uu_q) < \eps/2.
\]
Thus $d(\ww(p),\ww(q))<\eps$.
\end{proof}

\begin{theorem}[Interpolation Theorem]\label{interpolation-theorem}
There is a closed region $\Omega$ in $N$ with the following properties:
\begin{enumerate}
\item $\Omega$ is a $C^1$ manifold-with-boundary.
\item $\Omega$ contains $\{p: d(p,X)\le R\}$ and is disjoint from $\{p: d(p,Y)<R\}$.
\end{enumerate}
Furthermore,  $\Sigma:=\partial\Omega$ is uniformly $C^1$ near $\partial K$ in the following sense.
For every $\eps>0$, there is a $\delta>0$ such that if $p,q\in \Sigma$, if $d(p,\partial K)<\delta$,
 and  if $d(p,q)<\delta$, then 
\[
   d(\nu(p),\nu(q))< \eps,
\]
where $\nu(\cdot)$ is the unit normal to $\Sigma$ that points out from $\Omega$.
\end{theorem}

\begin{proof}
Let $0<\rho_i<R$ with $\rho_i\to R$ and let $h^{\rho_i}$ be as the harmonic
function defined at the beginning of \S\ref{interpolation-section}.
By passing to a subsequence, we can assume that $h^{\rho_i}$ converges locally uniformly
in $K\setminus S$ to a harmonic function such that
\[
h(p)
=
\begin{cases}
0    &\text{if $x(p)=R$ and $y(p)>R$}, \\
1    &\text{if $y(p)=R$ and $x(p)>R$}.
\end{cases}
\]

Let $c\in [1/3,2/3]$ be a regular value of $h$ and also of each $h^{\rho_i}$.
Thus
\[
  \Omega_i:=\{h^{\rho_i}\le c\}\cup \{x(\cdot)\le \rho_i\}
\]
is a $C^1$ manifold-with-boundary, the boundary being
\[
   \Sigma_i := \{p:  h^{\rho_i}(p)=c\}.
\]

By passing to a further subsequence, we can assume that $\Omega_i$ and $\Sigma_i$
converge to closed sets $\Omega$ and $\Sigma$.  (Actually, convergence holds
 without passing to a further subsequence, but we do not need that fact.)
Clearly, $\Sigma\setminus S= h^{-1}(c)$, and
\begin{equation}\label{nested}
   h^{-1}(c) \subset \Sigma \subset h^{-1}(c) \cup S.
\end{equation}
Note that each path from a point $p$ with $x(p)<R$ to a point $q$ with $y(q)<R$
must cross $\Sigma^i$ for all sufficiently large $i$ and thus must also cross $\Sigma$.
It follows that $S\subset \Sigma$ and therefore (by~\eqref{nested}) that
\[
    \Sigma = h^{-1}(c) \cup S.
\]
The convergence of $\Sigma_i \setminus S$ to $\Sigma\setminus S$
is smooth with multiplicity $1$ since $c$ is a regular value of $h$.
Note that
\[
  \nu_i = \nabla h^{\rho_i}/ |\nabla h^{\rho_i}|
\]
is the unit normal to $\Sigma_i$ that points out of $\Omega_i$.

\begin{claim}\label{convergence-claim} If $p_i\in \Sigma_i$ converges to $p\in \Sigma$,
then $\nu_i(p)$ converges to a limit $\nu(p)$.
\end{claim}

\begin{proof}[Proof of Claim~\ref{convergence-claim}]
If $p\in \Sigma\setminus S$, then this is true since $c$ is a regular value
of $h$.  (In this case, $\nu(p)= \nabla h(p)/ |\nabla h(p)|$.)

Now suppose that $p\in S$.
By Corollary~\ref{unique-u-corollary}, there is a unique $\uu(p)$ in $\Xx(p)$.
Let $\uu_i\in \Xx(p_i)$.  
Then $d(\nu_i(p_i), \uu_i) \to 0$ by Lemma~\ref{harmonic-lemma},
and $\uu_i\to \uu(p)$ by Corollary~\ref{u-v-corollary},
so $\nu_i(p)$
converges to $\uu(p)$.
Thus Claim~\ref{convergence-claim} is proved.
\end{proof}

Thus $\Sigma_i$ converges to $\Sigma$ in $C^1$ with multiplicity $1$.
(To see that the multiplicity is $1$, 
 note that if the convergence were multi-sheeted at a point $p\in \Sigma$,
then on some of the sheets, $\nu_i$ would converge to $\nu(p)$,and, on the other
sheets, $\nu_i$ would converge to $-\nu(p)$.)

Finally, the $(\eps, \delta)$ bounds on $\Sigma$ follow from the corresponding
bounds on the $\Sigma_i$.
\end{proof}

\section{Separating Flows}\label{separating-section}

Suppose $t\in [0,\infty)\mapsto M(t)$ is an $m$-dimensional integral Brakke flow in $N$.
We say that $M(\cdot)$ is a {\bf separating flow} provided:
\begin{enumerate}
\item $M(0)$ is the Radon measure associated to an embedded, $C^1$ manifold $\Sigma$.
\item $\Sigma$ is the boundary of a region $\Omega$.
\item if $p$ and $q$ are points in $N\setminus \Sigma$ with $p\in \Omega$ and $q\in \Omega^c$,
then any path in $N\times[0,\infty)$ from $(p,0)$ to $(q,0)$ must intersect the spacetime
support $\MMM$ of $M(\cdot)$.
\item The flow $M(\cdot)$ has the local regularity property of~\cite{white-local}.
\end{enumerate}

\begin{theorem}\label{separating-limits}
Suppose that $t\in [0,\infty)\mapsto M_i(t)$ is a sequence of  separating Brakke flows
that converge to a Brakke flow $M(\cdot)$.  
Suppose also that $\spt M(0)$ is a $C^1$ submanifold, and that $\spt M_i(0)$
converges in $C^1$ with multiplicity $1$ to $\spt M(0)$.
Then $M(\cdot)$ is a separating flow.
\end{theorem}

The theorem follows easily from the fact that $\MMM_i$ converges to $\MMM$ (which follows,
for example, from Brakke's clearing out lemma), and from the local regularity theory
  in~\cite{white-local}.

\begin{theorem}\label{separating-flow-existence-theorem}
Let $\Omega$ be a region in $N$ such that $\Omega$ is an $(m+1)$-dimensional
manifold-with-boundary of class $C^1$.   Let $\Sigma=\partial \Omega$.
Then there is a separating flow $M(\cdot)$ such that $M(0)$ is the Radon measure associated
to $\Sigma$.
\end{theorem}

This follows easily from elliptic regularization~\cite{ilmanen-elliptic}
 and Theorem~\ref{separating-limits}.
For completeness, we sketch the proof below

\begin{proof}
By passing to the double cover, we can assume that $N$ is oriented.
Consider first the case when $\Sigma$ is compact.
Let $\tilde g$ denote the product metric on $N\times [0,\infty)$.
For $\lambda>0$, let $A^\lambda$ be an $(m+1)$-chain in $N\times [0,\infty)$ that
minimizes mass with respect to the Riemannian metric
\[
   g^{\lambda}_{ij}(x,z) := e^{-(2\lambda/m)z} \tilde g_{ij}(x,z)   \qquad(x\in N, z\in [0,\infty))
\]
among all $(m+1)$-dimensional locally integral currents $A$ with boundary $[\Sigma]$.
Note that if $p$ is in the interior of $\Omega$ and and if $q$ is in $\Omega^c$,
then $(p,0)$ and $(q,0)$
lie in different components of $W:=(N\times[0,\infty))\setminus \spt A^\lambda$.
It follows that $(p,t)$ and $(q,t)$ lie in different components of $W$ for small $t>0$.

Now let $A^\lambda(t)$ be the image of $\Sigma$ under  translation by
  $(x,z) \mapsto (x, z - \lambda t)$, followed by restriction to $N\times (0,\infty)$.
For $t\ge 0$, let $M^\lambda(t)$ be the Radon measure on $N\times (0,\infty)$
associated to $A^\lambda(t)$.
Then $M^\lambda(\cdot)$ is a separating Brakke flow in $N\times (0,\infty)$.
By the theory of elliptic regularization~\cite{ilmanen-elliptic},
 the flows $M^\lambda(\cdot)$ converge 
 (after passing to a subsequence) as $\lambda\to\infty$
  to a limit Brakke flow $M^\infty(\cdot)$,
and  there is $m$-dimensional Brakke flow $M(\cdot)$
in $N$
such that $M^\infty(\cdot)$ is obtained from taking the product of $M(\cdot)$ with $(0,\infty)$.
By  Theorem~\ref{separating-limits}, the flow $M^\infty(\cdot)$ is separating.  Thus $M(\cdot)$ is separating.

For the general case of a noncompact $\Sigma$, we can find a sequence of compact $\Sigma_i$
converging to $\Sigma$ in $C^1$ with multiplicity $1$.
By the compact case of Theorem~\ref{separating-flow-existence-theorem},
  there is a separating flow $M_i(\cdot)$ associated to $\Sigma_i$.
By passing to a subsequence, we can assume that the $M_i(\cdot)$ converge
to an integral Brakke flow $M(\cdot)$.
By Theorem~\ref{separating-limits}, $M(\cdot)$ is also a separating flow.
\end{proof}

\begin{lemma}\label{separating-lemma}
Suppose $t\in [0,\infty)\mapsto M(t)$ is a separating flow in $\RR^{m+1}$ such that $M(0)$ is a hyperplane
$\Sigma$ with multiplicity~$1$.  Then $M(t)=M(0)$ for all $t\ge 0$.
\end{lemma}

\begin{proof}
Shrinking sphere barriers show that the spacetime support $\MMM$ of $M(\cdot)$ is
contained in $\Sigma\times [0,\infty)$.    By the separating property, $\MMM$ must
be all of $\Sigma\times[0,\infty)$.   Since the flow is an integral Brakke flow, it follows
each $M(t)$ is $\Sigma$ with multiplicity~$1$.
\end{proof}

\begin{comment}
\begin{theorem}
Suppose $\Omega\subset N$ is a closed region bounded by a $C^1$ hypersurface $\Sigma$.
Suppose that $t\mapsto M(t)$ is separating flow with $\spt M(0)=\Sigma$
and that $t\in [0,\infty)\mapsto X(t)$ is a weak set flow.
Let $C$ be the closed region in $N\times[0,\infty)$ bounded by $\Omega\times\{0\}$ and $\MMM$.
If $X(0)$ is disjoint from $\Omega$, then $X(t)$ is disjoint from
\[
   C(t):= \{x: (x,t)\in C\}
\]
for all $t\ge 0$.  Likewise, if $X(0)$ is contained in the interior of $\Omega$, then $X(t)$ is contained
in the interior of $C(t)$ for all $t$.
\end{theorem}
\end{comment}

%  The following section is required by Calc Var PDE.
%\begin{comment}
\section*{Data Availability}

Data sharing is not applicable to this article as no datasets were generated or analyzed during the current study.
%\end{comment}

\nocite{pedrosa-ritore}
\nocite{hoffman-wei}
\newcommand{\hide}[1]{}

\end{document}